\documentclass[a4paper,11pt]{article}
\usepackage[utf8]{inputenc}
\usepackage{amsmath,amstext,amsthm,amsfonts}
\usepackage{enumerate}

\newtheorem{theorem}{Theorem}[section]

\newtheorem{corollary}{Corollary}[section]
\newtheorem{lemma}[theorem]{Lemma}
\DeclareMathAlphabet{\mathpzc}{OT1}{pzc}{m}{it}
\newtheorem{definition}[theorem]{Definition}

\newtheorem{conjecture}[theorem]{Conjecture}

\DeclareMathOperator{\tr}{tr}

 \title{On the interplay between CPE metrics, vacuum static spaces and $\sigma_2$-singular spaces}
\author{Maria Andrade\footnote{The author was partially supported by PNPD/CAPES/Brazil. This work was written when the author was visiting the UFMG.}}
\date{}

\begin{document}
\maketitle

\begin{abstract}
We call CPE metrics the critical points of the total scalar curvature functional restricted to the space of metrics with constant scalar curvature of unitary volume. In this short note, we give a necessary and sufficient condition for a CPE metric to be Einstein in therms of $\sigma_2$-singular spaces. Such a result improves our understanding about CPE metrics and Besse's conjecture with a new geometric point of view. Moreover, we prove that the CPE condition can be replaced by the related vacuum static space condition to characterize closed Einstein manifolds in terms of $\sigma_2$-singular spaces.

% 
% Besse, \cite{besse2007einstein}, conjectured that, on a compact $n$-diemnsional orientable manifold, a critical ppint of the total scalar curvature restricted to the space of constant scalar curvature metrics of unit volume is Einstein. In this short note, we prove that this conjecture holds if the manifold is a singular space (\cite{silvaandrade2018}).
\end{abstract}

s\-ta\-tic space

\vspace{0.5cm}

\section{Introduction}
\label{intro}

Scalar curvature appears in the study of Einstein metrics, since this is a type of curvature closed related with the Ricci tensor. Recall that a Riemannian manifold is said to be Einstein if the Ricci tensor is multiple of the metric $g,$ i.e., $Ric_g=\lambda g,$ where $\lambda:M\to \mathbb{R}.$  In others words, $(M^n, g)$ is Einstein if its traceless tensor 
$$\mathring{Ric_g}=Ric_g-\dfrac{R_g}{n}g$$
is identically zero, where $Ric_g$ and $R_g$ are Ricci and scalar curvatures, respectively.

 Let $(M^n,g)$ be an $n$-dimensional closed (compact without boundary) oriented manifold with $n\geq 3,$ $ \mathcal{M}$ be the Riemannian metric space and $S_2(M)$ be the space of symmetric 2-tensors on $M.$ Fischer and Marsden \cite{fischer1975deformations} consider the scalar curvature map $\mathcal{R}:\mathcal{M}\to C^{\infty}$ which associates to each metric $g\in\mathcal{M}$ its scalar curvature. If $\gamma_g$ is the linearization of the map $\mathcal{R}$ and $\gamma^*_g$ is its $L^2$-formal adjoint, then they proved  that
\begin{eqnarray*}
 \gamma_gh=-\Delta_g\tr_gh+\delta^2_gh-\langle Ric_g,h\rangle
\end{eqnarray*}
and
\begin{eqnarray*}
 \gamma_g^*f=\nabla^2_gf-(\Delta_gf)g-fRic_g
\end{eqnarray*}
where $\delta_g=-div_g,$ $h\in S_2(M)$ and $\nabla^2_g$ is the Hessian form on $M^n$, respectively. 

In their study of the surjectivity of the scalar curvature map $\mathcal R$, Fischer and Marsden considered the so-called vacuum static equation $\gamma^*_g(f)=0$. We use the following definition.

  \begin{definition} Let $(M^n,g)$ be a complete Riemannian manifold. We say that $(M^n,g)$ is a vacuum static space if there is a nonzero smooth function $f$ on $M$ solving the following vacuum static equation
\begin{eqnarray}\label{eqvacu}
\gamma_g^*f=\nabla^2f-\Delta_gfg-Ric_g f=0.
\end{eqnarray}
We also refer $(M^n,g,f)$ as a vacuum static space.
 \end{definition}
 
In the last few decades, numerous investigations have been made on these spaces. The classification problem is a fundamental question as well as rigidity result. See \cite{ambrozio2017static}, \cite{yuan2015geometry} and the references contained therein.

 The Einstein-Hilbert functional $\mathcal{S}:\mathcal{M}\to\mathcal{R}$ is defined by
\begin{eqnarray*}\label{eq2}
 \mathcal{S}(g)=\displaystyle\int_MR_gdv_g.
\end{eqnarray*}

It is well known that the solution of the Yamabe problem shows that any compact Manifold $M^n$ admits a Riemannian metric with constant scalar curvature. In particular, the set $\mathcal{C}=\{g\in \mathcal{M}; R_g \ \text{is constant}\}\neq \emptyset.$ Thus, we can consider $\mathcal{M}_1=\{g\in \mathcal{M}; R_g\in \mathcal{C}\ \text{and $vol_g(M)=1$}\}\neq\emptyset.$

 Besse \cite{besse2007einstein} conjectured that the critical points of the total scalar curvature functional $\mathcal{S}$ restrited to $\mathcal{M}_1$ must be Einstein.  More precisely, the Euler-Lagrangian equation of Hilbert-Einstein action restricted to $\mathcal{M}_1$ may be written as the following critical point equation (CPE)
\begin{eqnarray}\label{CPE}
 \gamma_g^*f=\nabla^2_gf-(\Delta_gf)g-fRic_g=\mathring{Ric_g}.
\end{eqnarray}

In this setting, it is interesting to study the critical points of the restriction of the Einstein-Hilbert functional to $\mathcal{M}_1.$ Following the notations developed in \cite{barros2014critical} and \cite{neto2015note} we consider the following definition.

\begin{definition}
 A CPE metric is a triple $(M^n,g, f),$ where $(M^n,g)$ is a closed oriented Riemannian manifold of dimension $n\geq 3$ with constant scalar curvature and volume 1 and $f$ is a smooth potential satisfying equation \eqref{CPE}.
\end{definition}
 The Besse's conjecture (or CPE conjeture) can be rewritten as
 
\begin{conjecture}[\cite{besse2007einstein}]\label{besseconjec}
 A CPE metric is always Einstein.
\end{conjecture}

Note that the equation $\eqref{CPE}$ is equivalent to
\begin{eqnarray}\label{eqcpe1}
 \mathring{Ric_g}=\nabla^2_gf-\left(Ric_g-\dfrac{R_g}{n-1}\right)f
\end{eqnarray}
and
\begin{eqnarray}\label{eqcpen}
 (1+f)\mathring{Ric_g}=\nabla^2_gf+\dfrac{R_gf}{n(n-1)}g
\end{eqnarray}
Observe that if $f$ is a constant function and satisfies the equation (\ref{eqcpen}) then $f=0$ and this implies that $(M^n,g)$ is Einstein.
Moreover, if $(M^n,g,f)$ is a CPE metric with $f$ a non-constant function, then the set $$B=\{x\in M^n/f(x)=-1\}$$ has zero $n$-dimensional measure (see \cite{hwang2013three} and \cite{neto2015note}). Thus, to prove that CPE metric is Einstein is equivalent to show that $(g,f)$ satisfies the equation 
\begin{eqnarray}\label{eq5}
 \nabla^2_gf+\dfrac{R_gf}{n(n-1)}g=0
\end{eqnarray}
where $f$ is not a constant function. This implies that 
 \begin{eqnarray}\label{eqlaplacian}
 \Delta_gf=-\dfrac{R_g}{n-1}f.
\end{eqnarray}
Now, multiplying the equation \eqref{eqlaplacian} by $f$ and integrating over $M,$ we obtain 
\begin{eqnarray}\label{rpositive}
\dfrac{R_g}{n-1}\displaystyle\int_Mf^2dv_g=-\displaystyle\int_Mf\Delta_gfdv_g=\displaystyle\int_M|\nabla_gf|^2dv_g>0
\end{eqnarray}
since $f$ is not a constant function, so we get that $R_g>0.$

 If $f$ satisfies the equation \eqref{eq5}, then $(M^n, g)$ is isometric to $\mathbb{S}^n(r)$ where $r=\left(\frac{R_g}{n(n-1)}\right)^{1/2}$ (see \cite{tashiro1965complete}). In this way, considering $f$ non-constant, Conjecture \ref{besseconjec} can be rewritten changing Einstein property by the manifold being the canonical sphere.
% %
% Also, by integrating \eqref{eqlaplacian} over $M$ we get
% $$0=\displaystyle\int_M\Delta_gfdv_g=-\dfrac{R_g}{n-1}\displaystyle\int_Mfdv_g\implies \displaystyle\int_Mfdv_g=0.$$
% Thus, if $f$ satisfies \eqref{CPE}, then and it is a constant  function, then $f=0.$

Hwang \cite{hwang2013three} has proved the Besse's conjecture when $n=3$ and under the hypothesis $\ker (\gamma^*_g)\neq\{0\}.$  Barros and Ribeiro Jr. \cite{barros2014critical} proved that the Besse's conjecture is true for $4$-dimensional half locally conformally flat manifold. In \cite{barros2015critical} Barros et al proved that the conjecture \ref{besseconjec} is true for $4$-dimensional manifolds with harmonic tensor $W^+,$ where $W^+$ is called the self-dual part of the Weyl curvature tensor $W$. More recently, Neto \cite{neto2015note} proved a necessary and suficient condition for a CPE metric to be Einstein in terms of the potential function.

% That is, $f$ is an eigeinfunction of the Laplacian associated to the eigenvalue $R_g/(n-1).$ Since the Laplacian has non-positive spectrum we conclude that $R_g$ must be positive. On the other hand, by Lichnerowiz-Obata's Theorem (see [obata]), the first nonzero eigenvalue, $\lambda_1,$ the Laplacian satisfies the inequality $$\lambda_1\geq \dfrac{R_g}{n-1}.$$
% Thus, the equality holds in \eqref{eqlaplacian} if and only if $(M^n,g)$ is isometric to the round spheres with radius $r=\left(\frac{n(n-1)}{R_g}\right)^{1/2}.$  

In this note, we give a necessary and suficient condition for a CPE metric to be Einstein for $n\geq 3,$ improving the understanding about CPE metrics and Besse's conjecture with a new geometric point of view that involves the potential function $f$ provided by the CPE condition. More precisely, we prove the following result.

\begin{theorem}\label{main}
Let $(M^n,g,f)$ be an $n$-dimensional CPE metric with non-constant pontential function $f$ and $n\geq 3$. $(M^n,g)$ is Einstein if and only if $f\in \ker\Lambda_g^*$, where $\Lambda_g:S_2(M)\to C^{\infty}(M)$ is the linearization of the $\sigma_2$-curvature and $\Lambda^*_g$ is the $L^2$-formal adjoint of the operator $\Lambda_g$, i.e., $(M^n,g,f)$ is a $\sigma_2$-singular space.
\end{theorem}

An immediate consequence is the following

% In particular, we prove that Besse's conjecture  holds on a compact oriented n-dimensional manifold, $n\geq 3,$ under the condition that $ker \Lambda_g^*\neq 0,$ More precisely, we prove the following result

\begin{corollary}\label{cmain}
 Let $(M^n,g)$ be a closed, oriented  $n$-dimensional manifold, with $n\geq 3$. Let $(g,f)$ be a non-trivial solution of \eqref{CPE}. If $f\in \ker \Lambda_g^*$, then $(M^n,g)$ is isometric to the round sphere with radius $r=\left(\frac{n(n-1)}{R_g}\right)^{\frac{1}{2}}$ and $f$ is an eigenfunction of the Laplacian associated to the first eigenvalue $\frac{R_g}{n-1}$ on $\mathbb{S}^n(r)$. Hence $\dim\ker \Lambda_g^*=n+1$ and $\displaystyle\int_M fdv_g=0.$ 
\end{corollary}

Moreover, we prove that if $(M^n,g)$ is a closed Riemannian manifold and $\ker\Lambda_g^*\cap\ker\gamma_g^*\neq\{0\}$, then $(M^n,g)$ is an Einstein manifold. Thus, it is isometric to the standard sphere $\mathbb{S}^n.$ 

\begin{theorem}\label{t002} Let $(M^n,g,f)$ be an $n$-dimensional closed static vacuum space, $n\geq 3$. $(M^n,g)$ is Einstein if and only if the space $(M^n,g,f)$ is $\sigma_2$-singular. If $f$ is a non-constant function $(M^n,g)$ has to be isometric to the standard sphere $\mathbb{S}^n$, in the other case $(M^n,g)$ has to be Ricci flat.
\end{theorem}

Lin and Yuan has proved  in \cite{lin2016deformations} results about deformation of $Q$-curvature involving the kernel of $\Gamma^*_g,$ where $\Gamma_g:S_2(M)\to C^{\infty}(M)$ is the linearization of the $Q$-curvature and $\Gamma^*_g$ is the $L^2$-formal adjoint of the operator $\Gamma_g.$ In particular, they proved rigidity and others results, including an analogous result to our Theorem \ref{t002} in the context of $Q$-singular spaces.

We can summarize the above results in the following assertions.

\begin{corollary}\label{cor2}
Let $(M^n,g)$ be an $n$-dimensional closed oriented manifold with $n\geq 3$ and $f$ be a non-constant function defined in $M.$ We consider the following statements:

\begin{enumerate}[i)]
 \item $(M^n,g,f)$ is a CPE metric;
 \item $(M^n,g,f)$ is a vacuum static space;
 \item $(M^n,g,f)$ is a $\sigma_2$-singular space.
\end{enumerate}
If we choose any two statements we have that $(M^n,g)$ is an Einstein manifold, more specifically, $(M^n,g)$ is isometric to the standard sphere $\mathbb{S}^n$. In particular, if any two statements are true, the other one is also true.
\end{corollary}

In \cite{silvaandrade2018} was proved that if $(M^n,g)$ is a closed Einstein manifold with negative scalar curvature, then $(M^n,g)$ can not be a $\sigma_2$-singular space. In this way, if we  consider in the Corollary \ref{cor2} a more restrictive condition $iii)$, imposing that $R_g\neq0$, we can rewrite it as the following.

\begin{corollary}\label{cor3}
Let $(M^n,g)$ be an $n$-dimensional closed oriented manifold with $n\geq 3$ and $f$ be a non-constant function defined in $M.$ We consider the following statements:

\begin{enumerate}[i)]
 \item $(M^n,g,f)$ is a CPE metric.
 \item $(M^n,g,f)$ is a vacuum static space.
 \item $(M^n,g,f)$ is a $\sigma_2$-singular space with $R_g\neq0.$
 \item $(M^n,g)$ is an Einstein manifold.
\end{enumerate}
If we choose any two statements, we have that $(M^n,g)$ is isometric to the standard sphere $\mathbb{S}^n$. In particular, if any two statements are true, the others are also true.
\end{corollary}

% In fact, if we choose I) and II) is trivial. Now, if we choose I)and III) following by Theorem \ref{main}. Finally, II) and III) the result following by Theorem \ref{tvacuum}
% \begin{remark}
% 
%  If $(M^n,g)$ is an Einstein manifold, $\sigma_2$-singular space with $\sigma_2\neq 0,$ then $(M^n,g)$ is a vacuum static space.
%  If $(M^n,g)$ is an Einstein manifold and $(M^n,g,f)$ is a vaccum static space, then $(M^n,g,f)$ is a $\sigma_2$-singular space.
% \end{remark}

 If $(M^n,g)$ is a closed Einstein manifold with $R_g\neq0$, this result allow us to characterize those functions $f$ that appears both in CPE, vaccum static or $\sigma_2$-singular conditions, in fact for nonzero $f$ they are the same.

\section{Background and proofs}

We start this section recalling some important results  about $\sigma_2$-curvature. Then we present the proofs.

Let $(M^n,g)$ be an $n$-dimensional Riemanian manifold, $n\geq 3$. The $\sigma_2$-curvature, which will be denoted by $\sigma_2(g),$ is as a nonlinear map $\sigma_2:\mathcal{M}\to C^{\infty}(M),$ defined as the second elementary symmetric  function of the eigenvalue of the Schouter tensor $A_g=Ric_g-\frac{R_g}{2(n-1)}g$. In this case, we obtain that

\begin{equation}\label{eq000}
\sigma_2(g)=-\dfrac{1}{2}|Ric_g|^2+\dfrac{n}{8(n-1)}R_g^2.
\end{equation}

Motivated by works of Fischer and Marsden \cite{fischer1975deformations} and Lin and Yuan \cite{lin2016deformations}, in \cite{silvaandrade2018} was proved that the linearization of the $\sigma_2$-curvature at the metric $g$, $$\Lambda_g:S_2(M)\to C^{\infty}(M),$$ is given by
\begin{eqnarray*}
\Lambda_g(h) & = & \frac{1}{2}\left\langle Ric_g,\Delta_gh+\nabla^2tr_gh+2\delta^*\delta h+2\mathring{R}(h)\right\rangle\\
& & -\frac{n}{4(n-1)}R_g\left(\Delta_g tr_gh-\delta^2h+\langle Ric,h\rangle\right),
 \end{eqnarray*}
where $\delta^*$ is the $L^2$-formal adjoint of $\delta$ and $\mathring{R}(h)_{ij}=g^{kl}g^{st}R_{kijs}h_{lt}.$

Thus, its $L^2$-formal adjoint, $\Lambda_g^*:C^{\infty}(M)\to S_2(M)$, is 
\begin{eqnarray*}
%\label{Lambstar}
\begin{array}{rcl}
\Lambda_g^*(f) & = & \displaystyle\frac{1}{2}\Delta_g(fRic_g)+\frac{1}{2}\delta^2(fRic_g)g+\delta^*\delta(fRic_g)+f\mathring{R}(Ric_g)\\
& & -\displaystyle\frac{n}{4(n-1)}\left(\Delta_g(fR_g)g-\nabla^2(fR_g)+fR_gRic_g\right).
\end{array}
\end{eqnarray*}

This implies that
 \begin{equation}\label{eq003}
  \tr_g\Lambda_g^*(f)=\frac{2-n}{4}R_g\Delta_gf+\frac{n-2}{2}\langle \nabla^2f,Ric_g\rangle-2\sigma_2(g)f.
 \end{equation}
%Moreover, 

Note that,
 \begin{eqnarray}\label{eq004}
 \Lambda_g^*(1)&=&\frac{1}{2}\Delta_gRic_g-\frac{1}{4(n-1)}(\Delta_gR_g)g+\frac{2-n}{4(n-1)}\nabla^2R_g\nonumber\\
 &+&\mathring{R}(Ric_g)-\frac{n}{4(n-1)}R_gRic_g.
 \end{eqnarray}
 
Then, by (\ref{eq003}) and (\ref{eq004})  we obtain
\begin{equation}\label{eqn005}
\tr_g\Lambda_g^*(1)=-2\sigma_2(g)
\end{equation}
and
\begin{equation}\label{eqn006}
div_g \Lambda_g^*(1)=-\frac{1}{2}d\sigma_2(g).
\end{equation}

The relations (\ref{eqn005}) and (\ref{eqn006}) are similar to the relations between the Ricci tensor and the scalar curvature, namely $R_g=\tr_gRic_g$ and $div_gRic_g=\frac{1}{2}dR_g$. 

In \cite{silvaandrade2018} was introduced the notion of $\sigma_2$-singular space, which has the $L^2$-formal adjoint of the linearization of the $\sigma_2$-curvature map with nontrivial kernel, and under certain hypotheses it was proved rigidity and others results.

 \begin{definition}\label{def001}
A complete Riemannian manifold $(M,g)$ is $\sigma_2$-singular if 
$$ \ker \Lambda_g^*\not=\{0\},$$
 where $\Lambda_g^*: C^{\infty}(M)\to S_2(M)$ is the $L^2$-formal adjoint of $\Lambda_g$. We will call the triple $(M^n,g,f)$ as a $\sigma_2$-singular space  if $f$ is a nontrivial function in $\ker \Lambda_g^*.$
\end{definition}
% 
% On the other hand, Silva and Andrade, (see \cite{silvaandrade2018}), proved that if a $\sigma_2$-singular Einstein manifold has positive $\sigma_2$-curvature, then was proved the following rigidity result.

\begin{theorem}[\cite{silvaandrade2018}] \label{isometricsphere} 
Let $(M^n,g,f)$ be a closed $\sigma_2$-singular Einstein manifold with positive $\sigma_2$-curvature. Then $(M^n,g)$ is isometric to the round sphere with radius $r=\left(\frac{n(n-1)}{R_g}\right)^{\frac{1}{2}}$ and $f$ is an eigenfunction of the Laplacian associated to the first eigenvalue $\frac{R_g}{n-1}$ on $\mathbb{S}^n(r)$. Hence dim $\ker \Lambda_g^*=n+1$ and $\displaystyle\int_M fdv_g=0.$
\end{theorem}

The next Lemma is crucial for our result.

\begin{lemma}\label{mainlema}
 Let $(M^n,g,f)$ be an $n$-dimensional CPE metric, then  $$\tr \Lambda_g^*(f)=\left(\dfrac{n-2+nf}{2}\right)|\mathring{Ric}|^2.$$
\end{lemma}

\begin{proof}
Since $(M^n,g,f)$ is an $n$-dimensional CPE metric, then $f$ satisfies the equation \eqref{CPE}
\begin{eqnarray}\label{eqcpe}
\nabla^2_gf=\mathring{Ric_g}-\left(Ric_g-\dfrac{R_g}{n-1}\right)f.
\end{eqnarray}
 Thus, by equations \eqref{eqlaplacian}, \eqref{eq000} and \eqref{eqcpe}, we get
 \begin{eqnarray*}
 \tr_g\Lambda_g^*(f)&=&\frac{2-n}{4}R_g\left(\dfrac{-R_g}{(n-1)}\right)f+\frac{n-2}{2}\left< \mathring{Ric_g}-\left(Ric_g-\dfrac{R_g}{n-1}\right)f,Ric_g\right>\\
 &+&\left(|Ric_g|^2-\dfrac{n}{4(n-1)}R_g^2\right)f\\
 &=&\dfrac{n-2}{2}|\mathring{Ric_g}|^2+\dfrac{n}{2}\left(|Ric_g|^2-\dfrac{R_g^2}{n}\right)f\\
 &=&\left(\dfrac{n-2+nf}{2}\right)|\mathring{Ric_g}|^2.
\end{eqnarray*}
%Since $n\geq 3$ and $f$ is not a constant function.
This prove the result.
\end{proof}

% As direct consequence we get.
% 
% \begin{coring} Let $(M^n,g,f)$ be an $n$-dimensional CPE metric. Then the CPE conjecture is true if and only if $\tr_g\Lambda_g^*(f)=0.$
%  
% \end{coring}

\begin{proof}[Proof of Theorem \ref{main}]
We assume that $(M^n,g,f)$ is a CPE metric. First, we suppose that $(M^n,g)$ is Einstein, then by Lemma 2 in \cite{silvaandrade2018} we obtain that 
 \begin{equation*}\label{eq0002}
 \Lambda_g^*(f)=\frac{R_g(n-2)^2}{4n(n-1)}\left(\nabla^2f-(\Delta_gf)g-\frac{R_g}{n}fg\right).
\end{equation*}
So, by equation \eqref{eqlaplacian}, we get 
\begin{equation}\label{eq0002}
 \Lambda_g^*(f)=\frac{R_g(n-2)^2}{4n(n-1)}\left(\nabla^2f+\frac{R_g}{n(n-1)}fg\right).
\end{equation}
By hypothesis $(M^n,g)$ is Einstein, then using this in the equation \eqref{eqcpe} we obtain that the equation \eqref{eq5} is satisfied. Thus, by expression in \eqref{eq0002} we conclude that $f\in \ker\Lambda_g^*.$

Conversely, we assume that $(M^n,g,f)$ is a $\sigma_2$-singular space, i.e, $\Lambda_g^*(f)=0,$ in particular $\tr_g\Lambda_g^*(f)=0.$ Since $f$ is not a constant function and  $n\geq 3$, Lemma \ref{mainlema} implies that $|\mathring{Ric_g}|^2=0,$ thus $(M^n,g)$ is Einstein\footnote{Note that the condition $\Lambda_g^*(f)=0$ in the Theorem \ref{main} can be changed by the weaker one $\tr\Lambda_g^*(f)=0$.}
\end{proof}

\begin{proof}[Proof of Corollary \ref{cmain}]
 Let $(M^n,g, f)$ be an $n$-dimensional CPE metric. If $(M^n,g,f)$ is a $\sigma_2$-singular space, then by Theorem \ref{main}, we obtain that $(M^n,g)$ is Einstein, and in this case we obtain that $\sigma_2=\dfrac{(n-2)^2}{8n(n-1)}R_g^2>0.$ Thus, by Theorem  \ref{isometricsphere} $(M^n,g)$ is isometric to the round sphere with radius $r=\left(\frac{n(n-1)}{R_g}\right)^{\frac{1}{2}}$ and $f$ is an eigenfunction of the Laplacian associated to the first eigenvalue $\frac{R_g}{n-1}$ on $\mathbb{S}^n(r)$ with dim ker $\Lambda_g^*=n+1$ and $\displaystyle\int_M fdv_g=0.$\qed
\end{proof}

\begin{proof}[Proof of Theorem \ref{t002}] We assume that $(M^n,g,f)$ is a vacuum static space. First, we suppose that $(M^n,g)$ is Einstein, then $f$ satisfies the equation \eqref{eq0002} . Thus, $(M^n,g,f)$ is a $\sigma_2$-singular.  

Conversely, we assume that $(M^n,g,f)$ is a $\sigma_2$-singular space. By hypothesis the function $f$ satisfies the equation \eqref{eqvacu}. Thus, by \eqref{eq003} we obtain that
\begin{eqnarray*} 
\tr_g\Lambda_g^*(f)&=&\frac{2-n}{4}R_g\Delta_gf+\frac{n-2}{2}\langle \nabla^2f,Ric_g\rangle-2\sigma_2(g)f\\
&=& \frac{2-n}{4}R_g\Delta_gf+\frac{n-2}{2}\left\langle \left(Ric_g-\frac{R_g}{n-1}g\right) f ,Ric_g\right\rangle-2\sigma_2(g)f\\
&=&\frac{2-n}{4}R_g\Delta_gf+\frac{n-2}{2} \left(|Ric_g|^2-\frac{R_g^2}{n-1}\right)f-2\sigma_2(g)f,\\
\end{eqnarray*}

Taking the trace in the equation (\ref{eqvacu}) we have that 
$$\Delta_g f+\frac{R_g}{n-1}f=0$$ 
 and hence using \eqref{eq000} we get that
\begin{eqnarray*} 
\tr_g\Lambda_g^*(f)&=&-\frac{2-n}{4(n-1)}R_g^2f+\frac{n-2}{2} \left(|Ric_g|^2f-\frac{R_g^2}{n-1}f\right)-2\sigma_2(g)f,\\
&=&-\frac{1}{2}R_g^2f+\frac{n}{2}|Ric_g|^2f\\
&=&\frac{n}{2}\left(|Ric_g|^2-\frac{R_g^2}{n}\right)f\\
&=&\frac{n}{2}\left|Ric_g-\frac{R_g}{n}g\right|^2f.
\end{eqnarray*}

On the other hand, by hypothesis $f\in \ker\Lambda_g^*$, then
$$\frac{n}{2}\left|Ric_g-\frac{R_g}{n}g\right|^2f=0.$$
If $f$ is a nonzero constant function, then $(M^n,g)$ is Einstein and thus the $\sigma_2$-curvature is $\sigma_2=\dfrac{(n-2)^2}{8n(n-1)}R_g^2$. On the other hand, $\sigma_2=0$ due Theorem 1 in \cite{silvaandrade2018}, so $(M^n,g)$ is Ricci flat. If $f$ is non-constant, in a vacuum static space $df$ cannot vanish on $f^{-1}(0)$ (see the proof of the Theorem 1 in  \cite{fischer1975deformations}). Then $f^{-1}(0)$ is an $(n-1)$-dimensional submanifold of $M$, so $f$ vanishes only outside a dense subset of $M$, which implies that
\begin{eqnarray*}
Ric_g=\frac{R_g}{n}g.
\end{eqnarray*} 
 
Since $(M^n,g,f)$ is a vacuum static space and $f$ is a non-constant function, then $R_g>0$ and by equation \eqref{eq5} it follows that $(M^n,g)$ is isometric to the round sphere.

\end{proof}

\begin{proof}[Proof of Corollary \ref{cor2}] 
To get the conclusions we can choose conditions $i)$ and $ii)$, join equations \eqref{eqvacu} and \eqref{CPE} and invoke Theorem \ref{t002}, or we choose $i)$ and $iii)$ and use Theorem \ref{main} and Corollary \ref{cmain}, or we choose $ii)$ and $iii)$ and use Theorem \ref{t002}.
\end{proof}

\begin{proof}[Proof of Corollary \ref{cor3}] 
The combinations of conditions $i)$ and $ii)$, $i)$ and $iii)$ or $ii)$ and $iii)$ follow in the same way as in Corollary \ref{cor2}.

Using condition $iv)$ we can get the conclusions by combining with condition $i)$ and equation \eqref{eq5}, or combining with condition $ii)$ which give us condition $i)$ or, finally, combining with condition $iii)$ and using equation \eqref{eq0002} which give us equation \eqref{eq5}.
\end{proof}

%\backmatter

\section*{Acknowledgement}
%\vspace{0.5cm}
Maria Andrade thanks Ezequiel Barbosa, Pietro da Silva and Almir Silva Santos for their support and comments for this work. This work was done when Maria Andrade was visiting the Federal University of Minas Gerais in Belo Horizonte, supported by the grant PNPD/CAPES in 2019 and 2020. Maria Andrade would like to express her appreciation to its sponsorship and hospitality.

\bibliographystyle{abbrv}

\bibliography{main.bib}

\begin{thebibliography}{10}

\bibitem{ambrozio2017static}
L.~Ambrozio.
\newblock On static three-manifolds with positive scalar curvature.
\newblock {\em Journal of Differential Geometry}, 107(1):1--45, 2017.

\bibitem{barros2015critical}
A.~Barros, B.~Leandro, and E.~Ribeiro~Jr.
\newblock Critical metrics of the total scalar curvature functional on
  4-manifolds.
\newblock {\em Mathematische Nachrichten}, 288(16):1814--1821, 2015.

\bibitem{barros2014critical}
A.~Barros and E.~Ribeiro~Jr.
\newblock Critical point equation on four-dimensional compact manifolds.
\newblock {\em Mathematische Nachrichten}, 287(14-15):1618--1623, 2014.

\bibitem{besse2007einstein}
A.~L. Besse.
\newblock {\em Einstein manifolds}.
\newblock Springer Science \& Business Media, 2007.

\bibitem{fischer1975deformations}
A.~E. Fischer, J.~E. Marsden, et~al.
\newblock Deformations of the scalar curvature.
\newblock {\em Duke Mathematical Journal}, 42(3):519--547, 1975.

\bibitem{hwang2013three}
S.~Hwang.
\newblock Three dimensional critical point of the total scalar curvature.
\newblock {\em Bulletin of the Korean Mathematical Society}, 50(3):867--871,
  2013.

\bibitem{lin2016deformations}
Y.-J. Lin and W.~Yuan.
\newblock Deformations of q-curvature i.
\newblock {\em Calculus of Variations and Partial Differential Equations},
  55(4):101, 2016.

\bibitem{neto2015note}
B.~L. Neto.
\newblock A note on critical point metrics of the total scalar curvature
  functional.
\newblock {\em Journal of Mathematical Analysis and Applications},
  424(2):1544--1548, 2015.

\bibitem{silvaandrade2018}
A.~Silva~Santos and M.~Andrade.
\newblock Deformation of the {$\sigma_2$}-curvature.
\newblock {\em Ann. Global Anal. Geom.}, 54(1):71--85, 2018.

\bibitem{tashiro1965complete}
Y.~Tashiro.
\newblock Complete riemannian manifolds and some vector fields.
\newblock {\em Transactions of the American Mathematical Society},
  117:251--275, 1965.

\bibitem{yuan2015geometry}
W.~Yuan.
\newblock {\em The geometry of vacuum static spaces and deformations of scalar
  curvature}.
\newblock PhD thesis, UC Santa Cruz, 2015.

\end{thebibliography}

 Universidade Federal de Sergipe, Departamento de Matem\'{a}tica, 49100-000, S\~ao Crist\'ov\~ao, SE, Brasil. 
 
E-mail adresss: maria@mat.ufs.br
\end{document}